\documentclass[12pt]{article}

\usepackage[margin=1in]{geometry}

\usepackage{diagbox}
\usepackage{mathtools}
\usepackage{bbm}
\usepackage{latexsym}
\usepackage{epsfig}
\usepackage{amsmath,amsthm,amssymb,enumerate}
\usepackage[usenames,dvipsnames]{xcolor}

\allowdisplaybreaks
\usepackage{color}
\def\red{\color{red}}
\def\blue{\color{blue}}

\def\nn{\nonumber}
\def\a{\alpha} \def\b{\beta} \def\d{\delta} 
\def\e{\varepsilon}    \def\g{\gamma}
\def\G{\Gamma}  \def\k{\kappa}
     \def\l{\lambda}
 \def\m{\mu} \def\n{\nu} 
  \def\s{\sigma} 
\def\t{\tau}

\def\TO{\Tilde{O}}

\newcommand{\mc}[1]{\mathcal{#1}}

\def\by{{\bf y}}

\def\bx{{\bf x}}

\newtheorem{theorem}{Theorem}
\newtheorem{lemma}[theorem]{Lemma}

\newtheorem{claim}{Claim}

\newcommand{\wh}[1]{\widehat{#1}}

\newcommand{\brac}[1]{\left(#1\right)}

\newcommand{\bfrac}[2]{\left(\frac{#1}{#2}\right)}
\def\half{\tfrac{1}{2}}

\newcommand{\rai}{\rightarrow \infty}
\newcommand{\ra}{\rightarrow}

\newcommand{\set}[1]{\left\{#1\right\}}

\def\es{\emptyset}

\def\E{\mathbb{E}}

\def\Pr{\mathbb{P}}

\newcommand{\ignore}[1]{}

\def\cG{{\mathcal G}}

\def\cX{{\mathcal X}}

\newcommand{\beq}[2]{\begin{equation}\label{#1}#2\end{equation}}

\def\nn{\nonumber}

\def\cG{\mathcal{G}}

\def\bd{{\bf d}}

\newcounter{rot}

\begin{document}
\author{Patrick Bennett\thanks{Research supported in part by Simons Foundation Grant \#426894.}\\Departement of Mathematics\\ Western Michigan University\\Kalamazoo MI 49008-5248  \and Colin Cooper\\Department of Informatics\\ King's College London\\London WC2B 4BG \and Alan Frieze\thanks{Research supported in part by NSF grant DMS1952285}\\Department of Mathematical Sciences\\Carnegie Mellon University\\Pittsburgh PA 15213}

\title{Rainbow Greedy Matching Algorithms}
\maketitle

\begin{abstract}
    We consider the problem of finding a large rainbow matching in a random graph with randomly colored edges. In particular we analyze the performance of two greedy algorithms for this problem. The algorithms we study are colored versions of algorithms that were previously used to find large matchings in random graphs (i.e. the color-free version of our present problem).
\end{abstract}

\section{Introduction}
In this short note, we discuss greedy algorithms for finding rainbow matchings in sparse random graphs. Thus we start with the random graph $G_{n,m},m=cn/2$ where $c>0$ is a constant and then color each edge uniformly at random (u.a.r.) from a set of colors $Q=[q]$. A set $S$ of edges is said to be {\em rainbow colored} if every edge in $S$ has a different color. The decision problem for whether a colored graph has a rainbow matching of size $k$ is NP-complete \cite{LP}. Here we discuss the efficacy of simple greedy algorithms for finding large rainbow matchings.

The color-free version of this problem has been studied in several previous papers. Dyer, Frieze and Pittel \cite{1} studied two greedy algorithms for finding large matchings. The first algorithm (Greedy Matching) repeatedly chooses an edge $\set{x,y}$ u.a.r., adds it to the current matching and deletes the vertices $x,y$. This continues until the remaining graph has no edges. They showed that with high probability (i.e. with probability tending to 1 as $n$ grows, henceforth abbreviated w.h.p.) this algorithm produces a matching of size asymptotic to $\tfrac n2\brac{1-\frac{1}{c+1}}$. They also considered a variation (Modified Greedy Matching) where the algorithm first chooses a vertex uniformly at random, and then chooses a uniform random incident edge and then updates the matching. This algorithm does slightly better than the first, it produces a matching of size asymptotic to $\tfrac n2\brac{1-\frac{\log(2-e^{-c})}{2c}}$.  Further improvements were
obtained by Karp and Sipser  \cite{KS}, using KSGreedy, a modification of Greedy Matching. KSGreedy chooses a random vertex of degree one, if there is one, and adds its incident edge to the matching; otherwise it chooses a random edge and adds it to the matching. Karp and Sipser studied this algorithm in $G_{n,p},p=c/n$ and showed that it produced a matching of asymptotically maximal size. The Karp-Sipser algorithm  was further studied for $G_{n,c/n}$ by Aronson, Frieze and Pittel \cite{2}, who showed that the algorithm found a matching within $O(n^{1/5}\log^{O(1)}n)$ of the maximum.

We will prove the following theorems which we prefix with a formal statement of the algorithms.:
\paragraph{Greedy Algorithm.}
Formally the algorithm proceeds as follows:
\vspace{-.15in}
\begin{tabbing}
xxxxx\= xxx\= xxx\= xxxx\= \kill
{\bf GREEDY} \+ \\[1.5ex]
{\bf begin} \+ \\
$M \leftarrow \emptyset $; \\
{\bf while} $E(G)\neq \emptyset$ {\bf do} \\
{\bf begin} \+ \\
\>  Choose $e=\{ u,v\} \in E$  u.a.r.; $M \leftarrow M\cup \{ e\}$;  \\
\> $V \leftarrow V \setminus \{ u,v\}$ ; $G \leftarrow G[V]$\\
\> $F\leftarrow \set{f \in E:c(f)=c(e)}$ ; $E\leftarrow E\setminus F$;\-\\
{\bf end}; \- \\
Output $M$ \\
{\bf end} \-
\end{tabbing}
\begin{theorem}\label{th1}
Suppose that $q=\k n$.
Let $\m$ denote the size of the matching produced by GREEDY, then following hold w.h.p.
\begin{enumerate}[(a)]
\item If $\k=1/2$ then $\m\sim \tfrac12\brac{1-\frac{1}{(2c+1)^{1/2}}}n$.
\item If $\k=\tfrac12(1+\e)$, where $|\e|>0$ then  $\m\sim \tfrac12\brac{1-\frac{1+O(\e)}{(2c+1)^{1/2}}}n$.
\item If $\k<1/2c$ and $c>5$ then $\k(1-e^{-c(1/2\k-2)})n< \m<  \k(1-e^{-c/2\k})n$.
\item If $\k\gg1$ then $\m\sim \tfrac12\brac{1-\frac{1}{2c+1+\e_\k}}n$ where $\e_k=\frac{\log(c+1)}{2\k-1}-\frac{c}{2\k}+O\bfrac{c}{\k^2}$.
\end{enumerate}
\end{theorem}

\paragraph{Modified greedy algorithm.} The modified greedy algorithm is formally described as
\begin{tabbing}
xxxxx\= xxx\= xxx\= xxxx\= \kill
{\bf MODIFIED GREEDY} \+ \\[1.5ex]
{\bf begin} \\
$M \leftarrow \emptyset $; \\
{\bf while} $E(G)\neq \emptyset$ {\bf do} \+ \\
{\bf begin} \+ \\
Choose $v\in V$  u.a.r.;\\
\>If $N(v) = \es$, $V \leftarrow V \setminus \{v\}$ ; \\
\>Else if $N(v) \ne \es$  choose  $u\in N(v)$  u.a.r.;\\
\>\qquad Let  $e=\{ u,v\}$; $M \leftarrow M\cup \{ e\}$;   $V \leftarrow V \setminus \{ u,v\}$ ; \\
\>\qquad $F\leftarrow \set{f \in E:c(f)=c(e)}$ ; $E\leftarrow E\setminus F$;\\
\>$G \leftarrow G[V]$;\-\\
{\bf end}; \- \\
Output $M$ \\
{\bf end}
\end{tabbing}
\begin{theorem}\label{th2}
Suppose that $q=\k n$.  Let $\m$ be the size of the matching produced by MODIFIED GREEDY.
Then the following hold w.h.p.
\begin{enumerate}[(a)]
\item  Let $N(\t)$ be the solution of
\[
N(\t)=1-2\t+ \int_0^\t  \exp\left\{ - \frac c\k N(\s)(N(\s)+\s+\k-1)\right\} \;d\s,
\]
and let $\t_0 \in [0,1]$ be the solution to $N(\t)=0$. Then
$\mu \sim (1-\t_0)n$.
\item  If $\k \ge 1/2$ then
\[
\mu \le \frac{c-1+e^{-c}}{2c-1+e^{-c}}\;n.
\]
\end{enumerate}
\end{theorem}
Numerical calculations suggest that the matching obtained by  MODIFIED GREEDY is significantly larger than by GREEDY. In the following we have fixed $\k=1/2$ so that we can use Theorem \ref{th1}(a).
\begin{center}
\begin{tabular}{|c|c|c|}\hline
 $c$&GREEDY&MODIFIED GREEDY\\ \hline
0.5&0.092&0.148\\ \hline
1.0&0.146&0.216\\ \hline
1.5&0.184&0.257\\ \hline
2.0&0.211&0.285\\ \hline
2.5&0.233&0.316\\ \hline
3.0&0.250&0.322\\ \hline
3.5&0.264&0.334\\ \hline
4.0&0.276&0.345\\ \hline
4.5&0.287&0.355\\ \hline
5.0&0.296&0.361\\ \hline
\end{tabular}
\end{center}
{\bf Conjecture:} Given the above table we conjecture that if $\m_G$ and $\m_{MG}$ are the sizes of the matchings produced by GREEDY and MODIFDIED GREEDY respectively, then w.h.p. $\m_G\leq\m_{MG}$.

\section{GREEDY}
Let $G(t)$ denote the unmatched graph remaining after $t$ iterations, let $\nu(t)$ be the number of vertices, and  let $\m(t)$ denote the number of edges in $G(t)$.
At each step $t$, we choose a random edge { $\{x,y\}$}, add it to the matching $M(t)$, delete the vertices $x$ and $y$ from $V(t)$, and delete all edges of the same color as {$\{x,y\}$}.  Let $d_t(\cdot)$ denote degree in $G_t$. For the sake of our analysis we reveal the random graph and colors as we run the algorithm. More specifically, at each step $t$ we reveal our matching edge $e_t$ by choosing a random pair of distinct vertices.  Then for each of the other $\m(t)-1$ other edges $e'$ we reveal whether or not $e'$ shares an endpoint with $e$. Any $e'$ meeting $e$ is deleted. Conditional on the matching edge $e$ and the (say) $k$ deleted edges, the remaining edges comprise a uniform random set of $\m(t) -1 - k$ edges on the remaining set of $\nu(t)-2$ vertices.

A priori we do not know the degrees of any of the vertices. We just know that at step $t$ we have $\nu(t)$ vertices and a uniform random set of $\mu(t)$ edges. We reveal the location of one of these edges, which is equally likely to have any two distinct endpoints among the $\nu(t)$ vertices. We know there is an edge there just because we said we were revealing the location of an edge. Of course, after we reveal the location of that edge, we know that its two endpoints must have degree at least 1. But we only know that because we revealed the edge.

We reveal the color of $e_t$ by choosing a random color from among the unused colors. Finally we reveal any other edges of that same color and delete them. Thus we have

\begin{equation}\label{eqn:trend1}
   \E[\m(t+1)  \mid \m(t) ]=\bigg(\m(t)- \E(d_t(x)+d_t(y) -1 \mid \m(t))\bigg)\left(1-\frac{1}{q-t}\right).
\end{equation}

Note that  the number of vertices at step $t$ is $\nu(t)=n-2t$. We will assume (justified later) that
\[
\frac tn \le \min\left\{ \frac 12, \k \right\} - \Omega(1)
\]
so that $\nu(t), q-t \ge \Omega(n)$. Then
\begin{align*}
\E(d_t(x)\mid \m(t)) = \E(d_t(y)\mid \m(t))&= 1+(\n(t)-2) \cdot \frac{\m(t)-1}{\binom{\n(t)}{2}-1}\\
& = 1 + \frac{2(\m(t)-1)}{\n(t)+1}\\
& = 1 + \frac{2\m(t)}{\n(t)} + O\left( n^{-1}\right).
\end{align*}

So, picking up from \eqref{eqn:trend1} we have

\begin{align}
     \E[\m(t+1)\mid \m(t) ]&= \bigg(\m(t)- 1- \frac{4\m(t)}{\n(t)} + O\left( n^{-1}\right)\bigg)\left(1-\frac{1}{q-t}\right)\nonumber\\
     & = \m(t) - 1- \frac{4\m(t)}{\n(t)} - \frac{\m(t)}{q-t} + O\left( n^{-1}\right) \label{echange}
\end{align}
This leads us to consider the differential equation ($t=\t n,M(\t)=\m(t)/n$ here) which will simulate the process w.h.p.
\beq{dmdt}{
\frac{dM}{d\t}=-1-\frac{4M(\t)}{1-2\t}-\frac{M(\t)}{\k-\t}, \quad M(0)=\frac c2 \text{ where }\k=\frac{q}{n}.
}
We let $\t_0$ be the smallest positive root of $M(\t)=0$ where $M$ is the solution to the above initial value problem. We will show that w.h.p. the process ends with a rainbow matching with $\t_0 n + o(n)$ edges. Unsurprisingly, we will not come close to a perfect matching, nor will we come close to using every color. In particular we will see in Section \ref{sec:asymptotics} that for fixed $c, k$ we have
\[
\t_0 \le \min\left\{ \frac 12, \k\right\} - \Omega(1).
\]
To do this we will apply the following theorem of Warnke \cite{lutz}.

\begin{theorem}[Warnke Theorem 2 and Lemma 11]\label{thm:Lutz} Let $a, n>1$ be integers. Let $\mc{D} \subseteq \mathbb{R}^{a+1}$ be a connected and bounded open set. Let $(F_k)_{1 \le k \le a}$ be functions with $F_k:\mc{D} \rightarrow \mathbb{R}$. Let  $\mc{F}_0 \subseteq \mc{F}_1 \subseteq \ldots$ be $\sigma$-fields. Suppose that the random variables $((Y_k(t))_{1 \le k \le a}$ are nonnegative and $\mc{F}_t$-measurable for $t>0$. Furthermore, assume that, for all $t>0$ and $1 \le k \le a$, the following conditions hold whenever $(t/n, Y_1(t)/n, ..., Y_a(t)/n)\in \mc{D}$:
\begin{enumerate}[(i)]
   \item\label{trendLipschitz} $|\E[Y_k(t+ 1)-Y_k(t)| \mc{F}_t]-F_k(t/n, Y_1(t)/n, ..., Y_a(t)/n)|\le \delta$, where the function $F_k$ is $L$-Lipschitz-continuous on $\mc{D}$ (the `Trend hypothesis' and `Lipschitz hypothesis'),

   \item \label{boundedness} $|Y_k(t+1)-Y_k(t)| \le \theta$ and $\Pr(|Y_k(t+1)-Y_k(t)|>\phi \;|\; \mc{F}_t)\le \g$ (the `Boundedness hypothesis'), and that the following condition holds initially:

   \item \label{init} $\max_{1 \le k \le a}|Y_k(0)-\hat{y}_k n| \le \e n$ for some $(0,\hat{y}_1,\ldots, \hat{y}_a)\in \mc{D}$ (the `Initial condition').
\end{enumerate}
Suppose $R\in [1,\infty)$ and $T\in(0,\infty)$ satisfy $t \le T$ and $|F_k(z)| \le R$ for all $1 \le k \le a$ and $z=(\t, y_1, \ldots, y_a) \in \mc{D}$.
Then for $\e > (\delta + \g \theta)\min\{T, L^{-1}\}+R/n$, with probability at least
\begin{equation} \label{eqn:successprob}
   1-2 a \exp\left\{-\frac{n\e^2 }{8T \phi^2} \right\} - a Tn\g
\end{equation}
we have
\[
\max_{0 \le t \le \sigma n} \max_{1 \le k \le a}|Y_k(t)-y_k(t/n)n|<3\exp\{LT\}\e n
\]
where $(y_k(\tau))_{1 \le k \le a}$ is the unique solution to the system of differential equations $y'_k(\tau) =F_k(\tau, y_1(\tau), . . . , y_a(\tau))$ with $y_k(0) = \hat{y}_k$ for $1 \le k \le a$, and $\sigma=\sigma(\hat{y}_1, . . .\hat{y}_a)\in[0, T]$ is any choice of $\sigma>0$ with the property that $(\tau, y_1(\tau), . . . y_a(\tau))$ has $\L_\infty$-distance at least $3\exp\{LT\}\e$ from the boundary of $\mc{D}$ for all $\tau\in[0, \sigma)$.
\end{theorem}
The above theorem is a version of Theorem 5.1 of Wormald \cite{Wde}. We use Warnke's version here because of the condition \eqref{boundedness}. The probability bound given by Wormald's thoerem is not good enough for us if our variables could see a large one-step change. In particular, for our process our matching edge could have an endpoint with linear degree, causing the loss of a linear number of edges in a single step. However, since the probability of that event is so small, Warnke's version can handle it. Of course, Wormald \cite{Wde} discusses similar situations and describes how to handle them, but for us it is more convenient to use Warnke's version since it allows us to apply a single theorem as a black box.

We apply Theorem \ref{thm:Lutz} with  $a=1$ to the  random variable $Y_1(t)=\mu(t)$. We let
\[
\mc{D}=\left\{(\t, M): 0 \le \t \le \t_0, \; 0\le  M \le c \right\}.
\]
Now we make sure to satisfy condition \eqref{trendLipschitz}. By \eqref{echange} we can let $\d= O\brac{n^{-1}}$. From \eqref{dmdt} we have
\[
F_1(\t, M) = -1 -\frac{4M}{1-2\t}-\frac{M}{\k-\t}
\]
which is $L$-Lipschitz on $\mc{D}$ for $L=O(1)$ (here we use the fact that the denominators $1-2\t, \k-\t$ are bounded away from 0). We now move on to condition \eqref{boundedness}. We have to take $\theta=O(n)$ since it is possible (though very unlikely) to have a vertex of linear degree. We let $\phi = 2n^{0.1}$, and to find a suitable $\g$ we bound
\begin{align}
    \Pr(|\m(t+1)-\m(t)|>2n^{0.1} \;|\; \m(t)) &\le 2 \frac{\binom{\n(t)-1}{n^{0.1}} \binom{\binom{\n(t)}{2} - n^{0.1}}{\m(t) - n^{0.1}}}{\binom{\binom{\n(t)}{2}}{\m(t)}}\nonumber\\
    & =  2 \frac{\binom{\n(t)-1}{n^{0.1}} (\m(t))_{n^{0.1}}}{\left(\binom{\n(t)}{2} \right) _{n^{0.1}}}\nonumber\\
    & \le 2 \frac{ \left(\frac{\n(t)e}{n^{0.1}} \right)^{n^{0.1}} (\m(t))^{n^{0.1}}}{\left(\binom{\n(t)}{2} -n^{0.1}\right)^{n^{0.1}}}\nonumber\\
    & = 2 \left[ O\left(\frac{  \m(t)}{n^{0.1}  \n(t) } \right)\right]^{n^{0.1}}\nonumber\\
    &= \exp\{ - \Omega(n^{0.1})\} = \g.\label{eqn:gamma}
\end{align}
Now for condition \eqref{init}, any positive $\e$ will do (we will have to choose $\e$ more carefully later to satisfy future conditions). We can very comfortably choose $T= 1$, and $R=O(1).$ Now we choose
\[
\e = n^{-0.1} > (\delta + \g \theta)\min\{T, L^{-1}\}+R/n
\]
and so the probability bound in \eqref{eqn:successprob} goes to 1. Thus with high probability we have
\beq{sol1}{
\m(t)=M(\t)n+O(n^{0.9})
}
uniformly for $0\leq t\leq \t_0n+O\left(n^{0.9}\right)$, where $M(\t)$ is the solution to \eqref{dmdt} and $\t=t/n$. So w.h.p. the Greedy algorithm will produce a rainbow matching of $\t_0n+O(n^{0.9})$.

%

\subsection{Solution to the differential equation}

When $q=n/2$ (i.e. $\k=1/2$) the solution to \eqref{dmdt} is
\[
M(\t)=\frac{(2c+1)(1-2\t)^3-(1-2\t)}{4}.
\]
The smallest root of the above is
\begin{equation}\label{k=1/2}
\t_0=\frac{1}2-\frac{1}{2(2c+1)^{1/2}}.
\end{equation}

When $\k \neq 1/2$ we write \eqref{dmdt}  as
\[
\frac{dm}{d\t}+ M(\t) \brac{\frac{4}{1-2\t}+ \frac{1}{\k-\t}} =-1.
\]
Let $I=-2\log(1-2\t)-\log(\k-\t)$, then
\[
M(t)=e^{-I}\brac{B-\int e^Id\t},
\]
where   $B$  solves $M(0)=c/2$.
The solution for $\k=1/2$ is given above.

For $k \ne 1/2$, let $A=1/(2\k-1)^2$ then
\begin{align*}
\int e^I d\t=& \int \frac{1}{(\k-\t)(1-2\t)^2}\\
=&\int A \brac{\frac{1}{\k-\t}-\frac{2}{1-2\t}+\frac{4\k-2}{(1-2\t)^2}}\\
=& A\brac{-\log (\k-\t) +\log(1-2\t)+\frac{2\k-1}{1-2\t}}.
\end{align*}
Thus $B=c/2\k+ A((2\k-1)-\log \k)$ and for $\k \ne 1/2$
\beq{gen}{
M(\t)= \frac{(\k-\t)(1-2\t)^2  }{(2\k-1)^2}
\brac{\frac{c(2\k-1)^2}{2\k} +
	\log \frac{\k-\t}{\k(1-2\t)}-\frac{2(2\k-1)\t}{1-2\t}}.
}

\subsection{Asymptotics for $\t_0$ from $M(\t)=0$}\label{sec:asymptotics}
From \eqref{gen} let
\begin{equation}\label{fkt}
f_\k(\t)=\brac{\frac{c(2\k-1)^2}{2\k} +
	\log \frac{\k-\t}{\k(1-2\t)}-\frac{2(2\k-1)\t}{1-2\t}}.
\end{equation}
We consider the solution $\t_0$ to $f_\k(\t)=0$ in three cases.
\begin{enumerate}[{\bf C{a}se 1}]
\item As $\k \ra 1/2$, in Lemma \ref{k1/2} we show that  $\t_0$   tends to the value in \eqref{k=1/2}, and give the rate of convergence.
\item For $\k$ large, $M(\t)=0$ in \eqref{gen} satisfies $c=\frac{2\t}{1-2\t}+O(1/\k)$ which implies that w.h.p.  the Greedy algorithm will produce a rainbow matching of size $\sim \brac{\frac{1}2-\frac{1}{2(c+1)}+O(1/\k)}n$.
In Lemma \ref{krai} we give a detailed asymptotic  which also bounds $\t_0$ for finite $\k \ge 1$. Thus in the limit as $\k \rai$ we obtain a matching of size $nc/2(c+1)$; the value obtained without the coloring constraint.
\item For $\k\ll 1/2$, and fixed $c$ the solution  to $M(\t)=0$ in \eqref{gen}  satisfies $\t\sim \k \brac{1-e^{-c/2\k}}$. Lemma \ref{k->0} gives more detail.
\end{enumerate}
\begin{lemma} \label{k1/2}
If $\k=\half (1+\e)$, where $|\e|>0$,  the solution to $M(\t)=0$ in \eqref{gen} is $\t_0= \half\brac{ 1-\frac{1+O(\e)}{\sqrt{2c+1}}}$.
\end{lemma}
\begin{proof}
Put $h=(2\k-1)/(1-2\t)$ so that $\t=\tfrac12(1-(2\k-1)/h)$.
From \eqref{fkt},
\[
f_\k(h)=\a -h + \log(1+h), \qquad \text{where} \qquad \a=c\frac{(2\k-1)^2}{2\k}+\log \frac{1}{2\k}+(2\k-1).
\]
{\bf Case 1a: $2\k <1$.} \\
 We have $\t\leq\k$ and so $2\k\geq (1-(2\k-1)/h)$ or $2\k-1\geq -(2k-1)/h$. If $2\k-1<0$ then this implies that $0>h>-1$.

Then
\[
h-\frac{h^2}{2}+\frac{h^3}{2} \le \log (1+h) \le h-\frac{h^2}{2}.
\]
The lower bound term $h^3/2$ holds for $|h| \le 1/3$, by comparison with a geometric series. Thus
\[
\a-\frac{|h|^2}{2}-\frac{|h|^3}{2} \le 0 \le \a-\frac{|h|^2}{2},
\]
so that
\[
|h| \le \sqrt{2\a}, \qquad \text{and} \qquad |h| \ge \sqrt{\frac{2\a}{1+|h|}} \ge \sqrt{\frac{2\a}{1+\sqrt{2\a}}}.
\]

{\bf Case 1b: $2\k>1,0<h<1$.}\\
Then $h>0$, and now
\[
h-\frac{h^2}{2}\le \log (1+h) \le h-\frac{h^2}{2}+\frac{h^3}{2}.
\]
Note that these inequalities hold for all $h>0$.

So
\[
\a-\frac{h^2}{2} \le 0 \le \a-\frac{h^2}{2}+\frac{h^3}{2},
\]
leading to
\begin{equation}\label{h<1}
\sqrt{2\a} \le h \le \sqrt{\frac{2\a}{1-h}} \le  \sqrt{\frac{2\a}{1-\sqrt{2\a}}}.
\end{equation}
Finally
\begin{align*}
\a&= \e^2(c+\half) -\e^3(c+1/3)+O(c\e^4).\\
h&=(2\a)^{1/2}(1+O(\a^{1/2})=\e\sqrt{2c+1}(1+O(\e)).\\
\t_0&=\half \brac{1-\frac{\e}{h}}= \half\brac{ 1-\frac{1+O(\e)}{\sqrt{2c+1}}}.
\end{align*}
\end{proof}

\begin{lemma}\label{krai}
Assume $\k \ge 1$ and $c \ge c^*=(e-1)\bfrac{2\k}{2\k-1}^2$. The solution to $M(\t)=0$ in \eqref{gen} is
\[
\t_0= \frac12\brac{1-\frac{1}{(c+1)-\frac{c}{2\k}+\frac{\log(c+1)}{2\k-1}+O\bfrac{c}{\k^2}}}.
\]
\end{lemma}
\begin{proof}
Put $z=(\k-\t)/(\k(1-2\t))$ so that $\t=\k(z-1)/(2\k z-1)$. Provided $\k \ge 1/2$, $z \ge 1$ and $z \ra 1/(1-2\t)$ as $\k \rai$.  Then \eqref{fkt} becomes
\begin{equation}\label{fkz}
f_\k(z)=c\frac{(2\k-1)^2}{2\k}+2\k +\log z-2\k z,
\end{equation}
and $f_\k(z)=0$ iff
\begin{equation}\label{eqz}
z=\b+ \frac{\log z}{2\k} \qquad \text{where} \qquad \b=c\bfrac{2\k-1}{2\k}^2+1.
\end{equation}
For $\d<1$ put
\[
z=\b +\frac{(1+\d)}{2\k} \log \b,
\]
and equate. At equality \eqref{eqz} becomes
\begin{flalign*}
\b+\frac{(1+\d)}{2\k} \log \b =& \b+ \frac{1}{2\k} \log \brac{\b +\frac{(1+\d)}{2\k} \log \b}\\
=& \b+\frac{1}{2\k} \log \b  + \frac{1}{2\k} \log \brac{1+\frac{(1+\d)}{2\k\b} \log \b}.
\end{flalign*}
Thus
\[
\d \log \b= \log \brac{1+\frac{(1+\d)}{2\k\b} \log \b}.
\]
 The conditions $\d <1$, $\k \ge 1$ and $c \ge c^*$ imply that $\b \ge e$ so that $(\log \b)/\b\le 1/e$.
 Thus $\frac{(1+\d)}{2\k\b} \log \b <1$.
If $x<1$, then
\[
\frac{x}{1+x} \le \log (1+x) \le x,
\]
which (after canceling a $\log\b$ term)  means $\d$ must satisfy
\[
\frac{1+\d}{2\k\b+(1+\d) \log \b} \le \d \le \frac{1+\d}{2\k\b}.
\]
The RHS inequality is true if $\d \le 1/(2\k\b-1)$.

For $c\geq c^*$, $\log \b \ge 1$, and we can strengthen the LHS to
\[
\frac{1+\d}{2\k\b+(1+\d)} \le \d.
\]
This  reduces  to $1 \le 2\d\k\b +\d^2$, which is satisfied for $\d \ge 1/2\k\b$.

We conclude that
\[
z_0=\b +\frac{(1+\d)}{2\k} \log \b, \qquad \text{where} \qquad  \frac{1}{2\k\b} \le \d \le \frac{1}{2\k\b-1}.
\]
Thus
\begin{flalign*}
\t_0=& \frac{\k(z_0-1)}{2\k z_0-1}=\frac12 \brac{ 1-\frac{2\k-1}{2\k z_0-1}}\\
=& \frac12\brac{1-\frac{2\k-1}{2\k\b+(1+\d)\log \b-1}}\\
=&\frac12\brac{1-\frac{2\k-1}{c\frac{(2\k-1)^2}{2\k}+2\k-1+\frac{1+\d}{\d}\log\brac{1+\frac{1+\d}{2\k\b}\log\brac{c+1-\frac{1}{2\k}+\frac{1}{\k^2}}}}}\\
=& \frac12\brac{ 1-\frac1{c+1 -\frac{c}{2\k}+\frac{\log(c+1)}{2\k-1}+O\bfrac{c}{\k^2}}}.
\end{flalign*}
\end{proof}

\begin{lemma}\label{k->0} Let $\k <1/2c$ where $c>c_0=5$.
Let $\t_0$ be the smallest positive root of \eqref{fkt}.
Then
\[
\k(1-e^{-c/2\k+2c}) < \t_0 <  \k(1-e^{-c/2\k}).
\]
\end{lemma}
\begin{proof}

Let $\t=\k(1-e^{-c/2\k +\e})$, then
\begin{align*}
f_\k(\t)\;=& \;\frac{c}{2\k}-2c+2c\k +\log(e^{-c/2\k +\e}) +\log \frac 1{1-2\t}
+\frac{2\t(1-2\k)}{1-2\t}\\
\;=&\; \e-2c +2c\k + \log \frac 1{1-2\t} +\frac{2\t(1-2\k)}{1-2\t},
\end{align*}
where the last two terms on the RHS are positive.

Referring to \eqref{fkt} we see that $f_\k(0)=c(1-2\k)^2/2\k>0$  for any $\k <1/2$.
If $\e=0$ then assuming $\t<\k<1/2c$ and $c>c_0$ then
$f_\k(\t)<0$ so $\t>\t_0$, since $f_\k$ is monotone increasing in $\t$. On the other hand, if $\e \ge 2c$ then $f_\k(t)>0$ so $\t<\t_0$.
\end{proof}

\section{MODIFIED GREEDY}

Here we choose a random vertex $x$ and then a random neighbor $y$, add $\set{x,y}$, to the matching $M$ and delete the vertices of the edge from $V(t)$ and all edges of the same color as the edge. If $d_t(x)=0$ then we just delete $x$. Let $\n(t)=|V(t)|$. We let $q(t)$ denote the number of unused colors.
 Let $t_0$ be the number of steps when a single vertex $x$ of degree zero is deleted, and $t_1$ the number of steps when a pair of vertices $x,y$ are deleted. As $t=t_0+t_1$, $\nu(t)=n-t_0-2t_1$, and $q(t)=q-t_1$ it follows that
\begin{equation}\label{qt}
    q(t) = t + \n(t) +q -n.
\end{equation}

If we condition on $\m(t), \n(t)$, the remaining unrevealed random graph is uniform over all graphs with $\m(t)$ edges and $\n(t)$ vertices. Let
\[
\l =\l(t):= \frac{2\m(t)}{\n(t)}
\]
be the average degree of the unrevealed graph. Recall (see, for example, \cite{book}) that if $b=o(a^{2/3})$ then
\[
\binom ab \sim \frac{a^b}{b!} \exp\left\{-\frac{b^2}{2a} \right\}.
\]
Thus if $\m, \n$ are at least $n / \log \log  n$ (and observing they are at most $O(n)$)
then the probability the first chosen vertex has degree $k$ for $k \le \log^2 n$ is
\begin{align}
    \frac{\binom{\n-1}{k}\binom{\binom{\n-1}{2}}{\m-k}}{\binom{\binom{\n}{2}}{\m}}& =  \frac{(\n-1)^k}{k!} \frac{\binom{\n-1}{2}^{\m-k}}{(\m-k)!} \frac{\m!}{\binom \n2 ^\m} \exp\left\{ -\frac{k^2}{2(\n-1)}-\frac{(\m-k)^2}{2\binom{\n-1}{2}}+\frac{\m^2}{2\binom \n2}\right\}\nonumber\\
    & = \binom \m k \brac{\frac{\n-1}{\binom \n 2}}^k \brac{\frac{\binom{\n-1}2}{\binom \n2}}^\m \exp\left\{\TO\brac{\frac{1}{n}} + \frac{\m^2}{\n(\n-1)} \brac{ 1 - \frac{(1-k/\m)^2}{1-2/\n}}\right\}\nonumber\\
    & = \frac{\m^k}{k!} \cdot \brac{\frac{2}{\n}}^k \cdot \brac{1-\frac 2\n}^\m \exp\left\{\TO\brac{\frac{1}{n} + \frac{k\m}{\n^2} + \frac{\m^2}{\n^3}}\right\}\nonumber\\
    & = \frac{\l^k}{k!} \exp\left\{\brac{-\frac{2}{\n} + O\brac{\frac 1{\n^2}}}\m +  \TO\brac{\frac{1}{n}}\right\}\nonumber\\
    & = \frac{\l^k}{k!} e^{-\l} + \TO\brac{\frac{1}{n}} .\label{eqn:probdegk}
\end{align}
Conditional on the first chosen vertex having degree $k$ for $1 \le k \le \log^2 n$, the expected number of additional neighbors of the second vertex is
\begin{equation}\label{eqn:conddeg}
  \frac{2(\m-k)}{\n-1} = \l +\TO\brac{\frac{1}{n}}.
\end{equation}
Conditional on a total of $\le 2 \log^2 n$ edges being adjacent to the two chosen vertices, the expected number of edges deleted due to being the same color as the matching edge is
\begin{equation}\label{eqn:condcolor}
    \frac{\m - O(\log^2 n)}{q} = \frac{\m}{q} + \TO\brac{\frac{1}{n}}
\end{equation}
where the big-O term on the last line follows since (as we will see later) we have $q(t) = \Omega(n)$.

Let $\xi(t)=(\m(t),\n(t),q(t))$. Then (explanation follows)
\begin{align}
\E(\n(t+1)\mid \xi(t))& = \n(t)-2 + e^{-\l}+ \TO\brac{\frac{1}{n}} \label{eqn:ntrend}\\
\E(\m(t+1)\mid \xi(t))& =\m(t)-\l - (1-e^{-\l})\left( \l + \frac{\m(t)}{q(t)}\right) + \TO\brac{\frac{1}{n}}.\label{eqn:mtrend}
\end{align}
Indeed, for the first line note that we lose 2 vertices unless the first vertex has degree 0 (which happens with probability about $e^{-\l}$ by \eqref{eqn:probdegk}) in which case we lose only 1. For the second line, note that we expect to lose $\l$ edges adjacent to the first vertex. Then, if the first vertex has positive degree (probability about $1-e^{-\l}$ by \eqref{eqn:probdegk}), we expect to lose about $\l$ additional edges adjacent to the second vertex by \eqref{eqn:conddeg}, as well as about $\frac{\m(t)}{q(t)}$ edges of the same color as the matching edge by \eqref{eqn:condcolor}. The big-O term comes from the approximations \eqref{eqn:probdegk}, \eqref{eqn:conddeg} and \eqref{eqn:condcolor}.

Substituting $\nu/n=N, \mu/n=M$ and using \eqref{qt} with $\k=q(0)/n$,  $\t=t/n$, this leads to the equations
\begin{align}
\frac{dM}{d\t}&=\l(e^{-\l}-2)-\frac{M(1-e^{-\l})}{N+\t+\k-1}. \label{M'}\\
\frac{dN}{d\t}&=e^{-\l}-2, \label{N'}
\end{align}
where $\l=2M/N$.

If we substitute \eqref{N'} into \eqref{M'} it  gives the following.
\begin{flalign*}
M'=&\frac{2M}{N}N'+ M \frac{N'+1}{N+\t+\k-1}\\
\frac{M'}{M}=& 2\frac{N'}{N} +\frac{(N+\t+\k-1)'}{N+\t+\k-1} \qquad\text{on division by } M\\
\implies \log M=&2\log N+ \log (N+\t+\k-1)+ \log B\\
M=& B N^2 (N+\t+\k-1).
\end{flalign*}
Assuming $M(0)=c/2$ and $N(0)=1$ this gives
\[
M=\frac{c}{2\k}N^2 (N+\t+\k-1), \qquad \l= \frac{c}{\k}{ N(N+\t+\k-1)}.
\]
Interestingly, the above expression for $M$ is intuitive in a sense. In particular, it essentially says that the number of edges is approximately
\begin{align*}
     \m(t) &\approx  \binom {\n(t)}{2} \cdot \frac{c}{n}  \cdot \frac{q(t)}{q},
\end{align*}
i.e. the edge density in the remaining graph is about the original edge density $c/n$ times the probability that a given color is unused, $q(t)/q$.

Substituting this expression for $\l$ into \eqref{N'} gives
\begin{equation}\label{N'}
N' = e^{-\frac{c}{\k}{ N(N+\t+\k-1)}} - 2, \qquad N(0)=1.
\end{equation}
Let $\t_0$ be the smallest positive solution to $M(\tau)=0$, so either $N(\t_0)=0$ or $N(\t_0)+\t_0 + \k - 1=0$. We will show that $N(\t_0)=0$, i.e. we run out of vertices before we run out of colors. This should be unsurprising since we ought to never run out of colors (e.g. a positive proportion of colors simply never appear).

Indeed,  letting $Q(\t) = N(\t) + \t + \k - 1$, we have
\[
Q' = N' + 1 = e^{-\frac{c}{\k}{ NQ}} - 1  \ge -\frac{c}{\k} N Q
\]
and $Q(0) = \k > 0$ and so
\[
Q(\t) \ge \k e^{-\frac c\k \int_0^\t N(x)\; dx} > 0.
\]
Therefore $N(\t_0)=0$ and $Q(\t_0)>0$.

We will apply Theorem \ref{thm:Lutz} again. This time $a=2$ and our two random variables are $\nu(t)$ and $\m(t)$. We let
\[
\mc{D}=\left\{(\t, N, M): 0 \le \t \le \frac 12, \; \frac{1}{\log \log n} \le N \le 2,\; \frac{1}{\log \log n} \le  M \le c \right\}.
\]
Now we make sure to satisfy condition \eqref{trendLipschitz}. By \eqref{eqn:ntrend} and \eqref{eqn:mtrend} we can let $\d= \TO\brac{n^{-1}}$. The functions on the right hand sides of \eqref{M'} and \eqref{N'}
are $L$-Lipschitz on $\mc{D}$ for $L=O(\log \log n)$. We now move on to condition \eqref{boundedness}. We have to take $\theta=O(n)$ since it is possible (though very unlikely) to have a vertex of linear degree. We let $\phi = 2n^{0.1}$, and then the same calculation from \eqref{eqn:gamma} indicates we can choose $\g = \exp\{ - \Omega(n^{0.1})\}$.
Now for condition \eqref{init}, any positive $\e$ will do (we will have to choose $\e$ more carefully later to satisfy future conditions). We choose $T= 1$, $R=O(\log \log n),$
$\e = n^{-0.1} $. Thus with high probability we have
\beq{sol2}{
\m(t)=M(\t)n+\TO(n^{0.9}), \qquad \n(t) = N(\t)n+\TO(n^{0.9})
}
uniformly for $0\leq t\leq \t_0n+\TO\left(n^{0.9}\right)$. This implies that w.h.p. our process ends when it runs out of vertices, which happens after $\t_0 n + \TO(n^{0.9})$ steps. Now observe that at any step $t$ the number of remaining vertices $\n(t)$ is equal to $n-t-x$ where $x$ is the number of edges in the matching.
So when the process terminates we have $0=n-\t_0 n + \TO(n^{0.9}) - x$ and so the number of edges in the final matching $\mu$ is
\begin{equation}\label{mumu}
\mu =(1-\t_0)n+ \TO(n^{0.9}).
\end{equation}

\subsection{Upper bound on the value of $\m$}
Integrating \eqref{N'} subject to $N(0)=1$, we have
\begin{equation} \label{Nt}
N(\t)= 1-2\t+\int_0^\t e^{-\frac c\k N(N+\s+k-1)}\; d\s.
\end{equation}
Let $\t_0$ be the smallest positive solution to $N(\t)=0$. Assume  the result of the previous section that the algorithm terminates when this condition is met (asymptotically).
\begin{lemma}
Assuming $\k \ge 1/2$, the function $F(\t)=N(N+\t+\k-1)$ is convex in $[0,\t_0]$.
\end{lemma}
\begin{proof}
We first prove that $N''(\t) > 0$ for $\t\in [0,\t_0)$. From \eqref{N'} it is clear that $N'+1<0$ in this interval.
Thus
\[
N''=e^{-\frac c\k N(N+\t+\k-1)} -\frac c\k\cdot [N(N+\t+\k-1)]',
\]
where
\[
[N(N+\t+\k-1)]'=N'(N+\t+\k-1)+N(N'+1).
\]
We show below that this derivative is negative for $\k \ge 1/2$. The result that $N''>0$ follows from this. As $N>0$ and $N'+1<0$ we only need to show that $N+\t+\k-1 \ge 0$.\\
(i) As $\n(t)=n-2t_1-t_0=n-2t+t_0$, $\n(t) \ge n-2t$. So if $\t \le 1/2$ and $\k \ge 1/2$,
\[
N(\t)+\t+\k-1 \ge 1-2\t+\t+\k-1=\k-\t \ge 0.
\]
(ii) Also if $1/2 \le \t <  \t_0$, then as $N(\t)>0$,
\[
N+\t+\k-1 \ge \t+\k-1 \ge \t-\half \ge 0.
\]
Thus if
\begin{align}
F=& N(N+\t+\k-1)\nn,\\
F'=& N'(N+\t+\k-1) +N(N'+1) \le 0,\nn\\
F''=& N''(N+\t+\k-1)+2N'(N'+1) +NN'' \ge 0,\label{qq}
\end{align}
as all the terms on the RHS of \eqref{qq} are non-negative.
\end{proof}
Noting that $F$ is convex in $[0,\t_0]$ and that $F(0)=\k, F(\t_0)=0$, let $C(\t)=\k-\k \t/\t_0$ be the chordal line
of $F$ such that $F(\t)\le C(\t)$ in that interval. This implies that
\[
e^{-\frac c\k F(\t)} \ge e^{-\frac c\k C(\t)} = e^{-c+c\t/\t_0}.
\]
Inserting this into \eqref{Nt} we have
\[
N(\t) \ge 1-2\t+\int_0^\t e^{-c+c\s/\t_0} \; d\s.
\]
Put $\t=\t_0$ so that $N(\t)=0$ to obtain
\[
0\ge 1-2\t_0 +\frac{\t_0}{c} (1-e^{-c}) \qquad \implies \qquad \t_0 \ge \frac{1}{2-\frac 1c(1-e^{-c})}.
\]
From \eqref{mumu} it follows  asymptotically that
\[
\mu \le  \frac{1-\frac 1c(1-e^{-c})}{2-\frac 1c(1-e^{-c})}\; n.
\]

\section{Final thoughts}
We have made progress in understanding the performance of simple greedy algorithms for finding large rainbow matchings in sparse random graphs. Our result is precise for GREEDY in the case $\k=1/2$. This is not an easy question, given that it is related to finding large matchings in sparse random 3-uniform hypergraphs. Here an edge $\set{u,v}$ of color $c$ can be thought of as a triple $\set{u,v,c}$. Of course, the decision problem of whether a 3-uniform hypergraph has a matching of size $k$ is on Karp's famous list of 21 NP-complete problems \cite{Karp21}. 

We have reduced the analysis of algorithms GREEDY and MODIFIED GREEDY to the analysis of some differential equations. While these differential equations are not terribly complex, in general they lack an explicit elementary solution. We hope that further study will lead to a better understanding. We briefly looked at a colored version of the Karp-Sipser algorithm and constructed the appropriate differential equations. We will not pursue this line here, but instead wait until we can better understand these equations.

It would be nice to prove the conjecture, mentioned in Section 1, that MODIFIED GREEDY performs better than GREEDY. In the color-free setting, Dyer, Frieze and Pittel \cite{1} proved that the Modified Greedy Matching algorithm performs better than Greedy Matching on a random graph. However the rainbow version is complicated by the lack of explicit elementary solutions to the differential equations.

\end{document}